\documentclass[11pt]{amsart}

\scrollmode

\usepackage{times}
\usepackage{latexsym,bm}
\usepackage{amsmath,graphicx}
\usepackage{latexsym}
\usepackage{amscd,amsthm,amssymb}
\usepackage[all]{xy}

\newtheorem{thm}{Theorem}[section]
\newtheorem{prop}[thm]{Proposition}
\newtheorem{lem}[thm]{Lemma}
\newtheorem{cor}[thm]{Corollary}

\theoremstyle{definition}

\newtheorem{defn}[thm]{Definition}
\newtheorem{rem}[thm]{Remark}

\numberwithin{equation}{section}

\title[The minimal genus problem]{The minimal genus problem for elliptic surfaces}
\author{M.~J.~D.~Hamilton}
\address{      Institute for Geometry and Topology\\
               University of Stuttgart\\
               Pfaffenwaldring 57\\
               70569 Stuttgart\\
               Germany}
\email{mark.hamilton@math.lmu.de}
\date{\today}

\begin{document}

\begin{abstract}We solve a certain case of the minimal genus problem for embedded surfaces in elliptic 4-manifolds. The proofs involve a restricted transitivity property of the action of the orientation preserving diffeomorphism group on the second homology. In the case we consider we get the minimal possible genus allowed by the adjunction inequality.
\end{abstract}

\maketitle

\section{Introduction}

In their classical work from 1961, Kervaire and Milnor \cite{KM} showed that certain second homology classes in simply-connected 4-manifolds are not represented by embedded spheres. It therefore became an interesting question to find for a given homology class in a 4-manifold the minimal genus of an embedded closed connected oriented surface realizing that class. This question has been solved at least partly for rational and ruled surfaces and for 4-manifolds with a free circle action \cite{FV1, FV2, La, Li, LiLi1, LiLi2, LiLi3, LiLi4, R}. On symplectic 4-manifolds the question is related to the Thom conjecture \cite{KrMr, MSzT, OzSz}. In general, the adjunction inequality from Seiberg-Witten theory gives a lower bound on the genus of a surface representing a homology class in a closed oriented 4-manifold with a basic class and we can then ask if this lower bound is indeed realized. Usually the question is more tractable for classes of positive self-intersection and is still open in most situations in the case of negative self-intersections. In particular, it is still unknown whether there exist embedded spheres in the $K3$ surface of arbitrarily negative self-intersection. 

An interesting class of 4-manifolds are elliptic surfaces. We will restrict to minimal simply-connected elliptic surfaces with $b_2^+> 1$, but generalizations should be possible. It is natural to consider these 4-manifolds, because they are the second case in the Enriques-Kodaira classification of minimal simply-connected complex surfaces after the complex projective plane and the Hirzebruch surfaces, for which the minimal genus problem has already been solved \cite{KrMr, R}. The remaining third case are surfaces of general type. 

Note that every orientation preserving self-diffeomorphism of a closed oriented 4-manifold induces an isometry of the intersection form on the second homology (modulo torsion). A very useful fact is that for elliptic surfaces the image of the orientation preserving diffeomorphism group in the orthogonal group of the intersection form is known. This is due to Borcea, Donaldson and Matumoto \cite{B, D, Mat} for the $K3$ surface and to Friedman-Morgan and L\"onne in the general case \cite{FM, Lo}. We will combine this knowledge with the work of Wall on the transitivity of the orthogonal groups of unimodular quadratic forms \cite{W}. Similar to the case of rational surfaces, this will allow us to reduce the problem of representing a homology class by a minimal genus surface to certain special classes. We cannot treat the minimal genus problem in full generality. Instead we will concentrate on the first interesting special cases that come to mind. To state one of the results, we will prove the following in the special case that the elliptic surface has no multiple fibres, i.e.~is given by a surface $E(n)$ with $n\geq 2$:
\begin{thm} Let $X$ be an elliptic surface diffeomorphic to $E(n)$ with $n\geq 2$. Suppose $A$ is a non-zero class in $H_2(X;\mathbb{Z})$ orthogonal to the canonical class $K$ and of self-intersection $A^2=2c-2$ with $c\geq 0$. Then $A$ is represented by a surface of genus $c$ in $X$. This is the minimal possible genus.
\end{thm}
Note that the self-intersection number of classes orthogonal to the canonical class is always even, because the canonical class is characteristic. There is a similar, slightly more restrictive theorem in the case of elliptic surfaces with multiple fibres. We are also interested if we can realize homology classes by surfaces that are contained in certain nice neighbourhoods inside the elliptic surface, given by copies of an embedded Gompf nucleus $N(2)$.

\subsection*{Notations} In the following, $X$ will denote a minimal simply-connected elliptic surface with the complex orientation. By an elliptic surface we always mean a surface of this kind. Using the classification of elliptic surfaces \cite{GS}, $X$ is diffeomorphic to $E(n)_{p,q}$, where the coprime indices denote logarithmic transformations. We restrict to the case $n\geq 2$ or equivalently $b_2^+>1$; see \cite{La} for a discussion of Dolgachev surfaces $E(1)_{p,q}$ with $b_2^+=1$. All self-diffeomorphisms of $X$ are orientation preserving. We often denote a closed oriented surface in $X$ and the homology class it represents by the same symbol.

\section{Action of the diffeomorphism group}

Let $H_2(X)$ denote the integral second homology of $X$ and $\text{Diff}^+(X)$ the group of orientation preserving self-diffeomorphisms of $X$. The intersection form on second homology induces a unimodular quadratic form on the lattice $H_2(X)$. We denote by $O$ the orthogonal group of all automorphisms of $H_2(X)$ that preserve the intersection form. The elements of this group are called automorphisms of the intersection form. The action of diffeomorphisms on homology defines a group homomorphism $\text{Diff}^+(X)\rightarrow O$. 

There is a homomorphism $O\rightarrow\mathbb{Z}_2$, called the {\em spinor norm}, which is defined as follows. We can choose an orientation on all maximal positive definite linear subspaces of $H_2(X;\mathbb{R})$, cf.~\cite{Ray}: Fix any such subspace $U_0$ and let $\pi\colon H_2(X;\mathbb{R})\rightarrow U_0$ denote the orthogonal projection. The restriction of $\pi$ to any maximal positive definite subspace $U$ is an isomorphism with $U_0$. Choosing an orientation for $U_0$ we get an orientation for all maximal positive definite subspaces $U$ via $\pi$. This orientation varies continuously with $U$. The spinor norm of an element $\phi\in O$ is defined to be $\pm 1$ depending on whether $\phi$ preserves or reverses the orientation when mapping a maximal positive definite subspace of $H_2(X;\mathbb{R})$ to another one. A deformation argument shows that this does not depend on the choice of such a subspace. The subgroup of $O$ of elements of spinor norm $1$ is denoted by $O'$. 
\begin{defn}
We let $K$ denote the canonical class of $X$, which is minus the first Chern class. If $X$ is not the $K3$ surface and hence $K\neq 0$, let $k$ denote the Poincar\'e dual of $K$ divided by its divisibility. If $X$ is the $K3$ surface, let $k$ denote the class of a general fibre. In any case, $k$ is a primitive homology class of self-intersection zero. We also choose a second homology class $V$ such that $k\cdot V=1$. For example if $X$ has no multiple fibres we can choose for $V$ a section of an elliptic fibration. We denote by $O_k$ the automorphisms of the intersection form fixing $k$ and by $O_k'$ those of spinor norm $1$.   
\end{defn}
The next statement follows from Theorem 8 in \cite{Lo} due to L\"onne.
\begin{thm}[L\"onne]\label{thm Lonne}
The image of the diffeomorphism group $\text{Diff}^+(X)$ in $O$ is equal to $O'$ for the $K3$ surface and contains $O_k'$ for all other elliptic surfaces $X$.
\end{thm}
We now consider integral unimodular quadratic forms in general. We let $H$ denote the even hyperbolic form of rank $2$ and $E_8$ the standard positive definite even form of rank $8$. A {\em standard basis} for $H$ is a basis $e,f$ such that 
\begin{equation*}
e^2=0, f^2=0, e\cdot f=1.
\end{equation*}
Let $Q$ denote the quadratic form $Q=lH\oplus m(-E_8)$ with $l\geq 2$ and $m\in\mathbb{Z}$. In Theorem 6 in \cite{W}, Wall proved the following.
\begin{thm}[Wall]
The orthogonal group of $Q$ acts transitively on primitive elements of given square.
\end{thm}
We want to deduce the following.
\begin{prop}\label{Prop subgroup spinor norm 1}
The subgroup of elements of spinor norm 1 in the orthogonal group of $Q$ acts transitively on primitive elements of given square.
\end{prop}
We first prove the following lemma.
\begin{lem} For any even number $2a$ there exist primitive elements $p$ and $q$ of square $2a$ and automorphisms of $Q$ of spinor norm $+1$ and $-1$ which map $p$ to $q$.
\end{lem}
\begin{proof}
We consider $Q=lH\oplus m(-E_8)$ and let $e, f$ denote a standard basis for the first $H$ summand. Let $p=e+af$ and $q=-e-af$. Then $p$ and $q$ are primitive elements and $p^2=q^2=2a$. Consider the automorphism of $Q$ which is minus the identity on the first $H$ summand and the identity on all other summands and the automorphism which is minus the identity on the first two $H$ summands and the identity on all other summands. These automorphisms have spinor norm $-1$ and $+1$ and map $p$ to $q$.
\end{proof}
We now prove Proposition \ref{Prop subgroup spinor norm 1}.
\begin{proof}
Let $x$ and $y$ be arbitrary primitive elements of square $2a$ and let $p$ and $q$ be the elements from the lemma of the same square. By Wall's theorem there exist automorphisms in $O$ mapping $x$ to $p$ and $q$ to $y$. Choosing an automorphism that maps $p$ to $q$ of the correct spinor norm we get by composing an automorphism of spinor norm $+1$ mapping $x$ to $y$. 
\end{proof}
We now consider the elliptic surface $X$.
\begin{lem}
The self-intersection number $V^2$ is even if and only if $X$ is spin.
\end{lem}
\begin{proof}
The intersection form on the span of $k$ and $V$ is unimodular, hence it is unimodular on the orthogonal complement. The intersection form on this complement is even, since the canonical class $K$ is characteristic. The claim now follows because $X$ is spin if and only if the intersection form on both summands is even. 
\end{proof}
Let $V^2=2b$ in the spin case and $V^2=2b+1$ in the non-spin case. 
\begin{defn}
Define a homology class $W=V-bk$. Then the intersection form on the span of $k$ and $W$ is the form $H$ in the spin case and the form $H'$ given by
\begin{equation*}
H'=\left(\begin{array}{cc} 0 & 1 \\ 1 & 1 \end{array}\right)
\end{equation*}
in the non-spin case. Note that $H'$ is isomorphic to $\langle +1\rangle \oplus \langle -1 \rangle$.
\end{defn}
The complete intersection form of $X$ is then given by 
\begin{equation}\label{eqn int form}
Q_X=H\oplus lH\oplus m(-E_8)\text{ or } H'\oplus lH\oplus m(-E_8),
\end{equation}
depending on whether $X$ is spin or non-spin, where $l\geq 2$ since $b_2^+\geq 3$ and $m>0$. We also want to choose a standard basis for the second $H$ summand in the intersection form. We first consider the case of the $K3$ surface: It is known that the $K3$ surface contains a rim torus $R$ of self-intersection zero and a vanishing sphere $S$ of self-intersection $-2$ such that $R$ and $S$ intersect transversely in one positive point; see page 73 in \cite{GS}. Both arise from the fibre sum construction $K3=E(1)\#_{F=F}E(1)$ along a general fibre $F$, given by
\begin{equation*}
K3=(E(1)\setminus\text{int}\,\nu F)\cup_\psi(E(1)\setminus\text{int}\,\nu F), 
\end{equation*}
with fibred tubular neighbourhood $\nu F\cong S^1\times S^1\times D^2$ and gluing diffeomorphism $\psi$ on the boundary of the tubular neighbourhood. The gluing diffeomorphism preserves the splitting and is given by the identity on the torus and complex conjugation on $\partial D^2$. The rim torus in this construction is given by
\begin{equation*}
R=S^1\times \{\ast\}\times \partial D^2\subset \partial\nu F\subset K3.
\end{equation*}
The vanishing sphere is obtained by sewing together two vanishing disks of relative self-intersection $-1$ coming from elliptic Lefschetz fibrations on $E(1)$. These vanishing disks bound the vanishing cycles
\begin{equation*}
\{\ast\}\times S^1\times\{\ast\}\subset \partial\nu F
\end{equation*}
in each copy of $E(1)\setminus\text{int}\,\nu F$, that get identified under the gluing diffeomorphism. Recall the following definition from \cite{Gnuc}:
\begin{defn} The {\em nucleus $N(2)$} is defined as the 4-manifold with boundary given by the neighbourhood of a cusp fibre and a section of self-intersection $-2$ in $K3$.
\end{defn}
The nucleus contains also a smooth torus fibre homologous to the cusp. The second homology of the nucleus is isomorphic to $\mathbb{Z}^2$ and spanned by this torus and the section. In addition to the nucleus containing a fibre and a section given by the definition, the $K3$ surface contains two other embedded copies of $N(2)$, disjoint from the first one. The rim torus $R$ and the vanishing sphere $S$ are embedded in one such copy \cite{GM} and correspond to the fibre and the section. Since this nucleus is disjoint from a general fibre it is still contained in an arbitrary elliptic surface $X$ of the type above, because the fibre sums and the logarithmic transformations resulting in the manifold $X=E(n)_{p,q}$ are performed in the complement of the nucleus. We can also choose the surface representing the class $V$ to be disjoint from this nucleus. 
\begin{defn}
Let $T$ denote the torus of self-intersection zero obtained by smoothing the intersection between $R$ and $S$. Then $T$ represents the class $R+S$ and the classes $R$ and $T$ are a standard basis for the second $H$ summand in the intersection form of the elliptic surface $X$.
\end{defn}
Using Theorem \ref{thm Lonne} and Proposition \ref{Prop subgroup spinor norm 1} we have the following:
\begin{prop}\label{prop main}
Let $X$ be an elliptic surface and $B$ an arbitrary class in the subgroup $lH\oplus m(-E_8)$ of $H_2(X)$ as in equation \eqref{eqn int form}. Then we can map $B$ to any other class in $lH\oplus m(-E_8)$ of the same square and divisibility by a self-diffeomorphism of $X$. In particular, we can map $B$ to a linear combination of the classes $R$ and $S$. This diffeomorphism is the identity on the other summand $H$ or $H'$ of $H_2(X)$, given by the span of $k$ and $W$. Suppose $X$ is the $K3$ surface. Then we can map any class in $H_2(X)$ to any other class in $H_2(X)$ of the same divisibility and square using a self-diffeomorphism of $X$.
\end{prop}
The result for the $K3$ surface can also be found in \cite{KrMr0, La}. As a final preparation, we consider the following theorem on the adjunction inequality from Seiberg-Witten theory \cite{GS,KrMr,OzSz}:
\begin{thm} Let $Y$ be a closed oriented 4-manifold with $b_2^+>1$. Assume that $\Sigma$ is an embedded oriented connected surface in $Y$ of genus $g(\Sigma)$ with self-intersection $\Sigma^2\geq 0$, such that the class represented by $\Sigma$ is non-zero. Then for every Seiberg-Witten basic class $L$ we have
\begin{equation*}
2g(\Sigma)-2\geq \Sigma^2+|L\cdot\Sigma|.
\end{equation*}
If $Y$ is of simple type and $g(\Sigma)>0$, then the same inequality holds for $\Sigma\subset Y$ with arbitrary square $\Sigma^2$.
\end{thm}
A basic class is a characteristic class in $H^2(Y;\mathbb{Z})$ with non-vanishing Seiberg-Witten invariant. If $L$ is a basic class, then $-L$ is also a basic class. The basic classes of the elliptic surfaces $X=E(n)_{p,q}$ are completely known \cite{FS}. They are given by the set
\begin{equation*}
\{rk\mid r\equiv npq-p-q\bmod 2,\, |r|\leq npq-p-q\},
\end{equation*}
where $k$ is the primitive class as above. The canonical class of the elliptic surface $E(n)_{p,q}$ is given by
\begin{equation*}
K=(npq-p-q)k.
\end{equation*}
It follows that the basic classes are certain multiples of the class $k$, where the maximal values at the end are given (up to sign) by the canonical class $K$. Hence we get:
\begin{cor}
The adjunction inequality for the elliptic surfaces $X$ reduces to the statement that
\begin{equation*}
2g(\Sigma)-2\geq \Sigma^2+|K\cdot\Sigma|
\end{equation*}
for every embedded surface $\Sigma$ of genus $g(\Sigma)$, representing a non-zero class with self-intersection $\Sigma^2\geq 0$.
\end{cor}

\section{Minimal genus problem for the $K3$ surface}

The minimal genus problem for classes of non-negative square in the $K3$ surface has already been solved \cite{La}. In this section we recall this solution as a preparation for the general case. The $K3$ surface has canonical class $K=0$. Hence the adjunction inequality implies for the genus of a smooth surface $\Sigma$ that $2g(\Sigma)-2\geq\Sigma^2$ if the homology class represented by this surface is non-zero.

\begin{defn}
{\em The standard surface of genus $g$ embedded in the nucleus $N(2)$} is by definition the section of self-intersection $-2$ ($g=0$), the general fibre of self-intersection $0$ ($g=1$) or the surface of genus $g\geq 2$ and self-intersection $2g-2$ obtained by smoothing the intersection points of the section and $g$ parallel copies of the general fibre. These surfaces represent primitive homology classes.
\end{defn}
According to Proposition \ref{prop main} we can map any primitive class in the $K3$ surface via a self-diffeomorphism to any other primitive class of the same square. Hence every primitive class of self-intersection $2c-2$ with $c\geq 0$ is represented by a surface of genus $c$ inside some nucleus in $K3$. This is the minimal possible genus according to the adjunction inequality. In particular, every primitive homology class in the $K3$ surface of square zero is represented by the standard torus in a nucleus $N(2)$ inside $K3$. 

To solve the case of divisible classes with non-negative square we use Lemma 7.7 in \cite{KrMr1} due to Kronheimer-Mrowka (see also Lemma 14 in \cite{La}):
\begin{lem}[Kronheimer-Mrowka]\label{lem Lawson}
Suppose that $Y$ is a closed connected oriented 4-manifold. For an embedded surface $\Sigma$ let $a(\Sigma)=2g(\Sigma)-2-\Sigma^2$. If $h\in H_2(Y;\mathbb{Z})$ is a homology class with $h^2\geq 0$ and $\Sigma_h$ is a surface of genus $g$ representing $h$ and $g\geq 1$ when $h^2=0$, then for all $r>0$, the class $rh$ can be represented by an embedded surface $\Sigma_{rh}$ with
\begin{equation*}
a(\Sigma_{rh})=ra(\Sigma_h).
\end{equation*}
\end{lem}
We can apply the construction of this lemma to divisible classes of non-negative square inside the nucleus $N(2)$ to get surfaces that represent these classes in the nucleus (the construction in the proof of this lemma works in a tubular neighbourhood of $\Sigma_h$ and does not need the assumption that $Y$ is closed). In this case $a(\Sigma_h)$ is zero, hence also $a(\Sigma_{rh})$ is zero. We have:
\begin{cor}\label{stand div N(2)}
Every non-zero class in $H_2(N(2))$, not necessarily primitive, which has self-intersection $2c-2$ with $c\geq 0$ is represented by an embedded surface of genus $c$ in $N(2)$.
\end{cor}
\begin{defn}
We call the surfaces in the nucleus obtained by the construction preceding Corollary \ref{stand div N(2)} {\em standard}.
\end{defn}
The transitivity of the action of the diffeomorphism group then implies that every divisible class of non-negative square in $K3$ can also be represented by a standard surface inside a nucleus $N(2)$. Hence we get:
\begin{cor}\label{cor K3}
Consider the $K3$ surface. Every non-zero homology class of self-intersection $2c-2$ with $c\geq 0$ is represented by a surface of genus $c$. We can assume that it is embedded as the standard surface in a nucleus $N(2)$ inside $K3$. This is the minimal possible genus. 
\end{cor}
This result can be found in the paper \cite{La} due to Lawson, except for the observation that these surfaces of minimal genus can be realized inside an embedded nucleus $N(2)$ in the $K3$ surface.

\section{Minimal genus problem for other elliptic surfaces}\label{section min genus other elliptic}
We now consider the general case of minimal simply-connected elliptic surfaces $X$ with $b_2^+>1$. The adjunction inequality implies for surfaces $\Sigma$ orthogonal to $K$ again that $2g(\Sigma)-2\geq \Sigma^2$. The self-intersection of such a surface is even, because the canonical class is characteristic. Using Proposition \ref{prop main} and Corollary \ref{stand div N(2)} we get:
\begin{cor}
Let $X$ be an elliptic surface. Then every non-zero homology class $A$ of self-intersection $2c-2$ with $c\geq 0$ that is orthogonal to the classes $K$ and $V$ is represented by a surface of genus $c$. We can assume that it is embedded as the standard surface in a nucleus $N(2)$ inside the 4-manifold $X$. This is the minimal possible genus.
\end{cor}
\begin{proof}
The assumptions imply that $A$ can be mapped via a diffeomorphism to $A'=\gamma R+\delta S$. Since $R$ and $S$ are constructed in a nucleus $N(2)$ the claim follows.
\end{proof}
\begin{rem}
If we relax the assumption and only assume that $A$ is orthogonal to $K$, it seems that the surface is in general not contained in a nucleus $N(2)$. For example the general fibre is contained in a nucleus $N(n)_{p,q}$.
\end{rem}
We can deal with the case $A^2=-2$ in a slightly more general situation:
\begin{prop}
Let $X$ be an elliptic surface. Then any homology class $A$ orthogonal to $K$ and of self-intersection $-2$ is represented by the standard sphere in a nucleus $N(2)$ in the 4-manifold $X$.
\end{prop}
\begin{proof}
The assumptions imply that there exists a self-diffeomorphism of $X$ mapping $A$ to
\begin{equation*}
A'=\alpha k+S,
\end{equation*}
where $S$ is the vanishing sphere. Consider the following map $\phi$ on $H_2(X)$ which on the first two summands of the intersection form is given by
\begin{align*}
k&\mapsto k\\
W&\mapsto W+\alpha R\\
R&\mapsto R\\
S&\mapsto S-\alpha k
\end{align*}
and is the identity on all other summands. It is easy to check that $\phi$ is an isometry. Letting $\alpha$ be a real number and taking $\alpha\rightarrow 0$ we see that $\phi$ has spinor norm $+1$. Hence it is an element in $O'_k$ and therefore induced by a self-diffeomorphism. It maps $A'$ to $S$. This implies the claim. 
\end{proof}
\begin{rem}
This result should be compared to the fact that every class of square $-2$ in the complement of a general fibre in $X$ is represented by an embedded sphere \cite{FM, Lo}. 
\end{rem}

We now restrict to the case of elliptic surfaces without multiple fibres, i.e.~$X=E(n)$ with $n\geq 2$, because the following arguments seem to work only in this case. The class $k$ is represented by a general fibre $F$. We also have the rim torus $R$. Proposition \ref{prop main} implies:
\begin{lem}
If $A$ is a homology class orthogonal to $K$ and of self-intersection zero, then there exists a self-diffeomorphism of $X$ that maps $A$ to
\begin{equation*}
A'=\alpha F+\gamma R.
\end{equation*}
\end{lem}
We want to show that $A'$ can be represented by an embedded torus. The construction involves the circle sum from \cite{LiLi4}. The idea is the following: Let $\Sigma_0$ and $\Sigma_1$ denote two disjoint connected embedded oriented surfaces in a 4-manifold $Y$. We can tube them together in the standard way to get a surface of genus $g(\Sigma_0)+g(\Sigma_1)$. Sometimes, however, we can perform a different surgery that results in a surface of smaller genus. Let $S^1_i\subset \Sigma_i$ denote embedded circles that represent non-trivial homology classes in the surfaces. In each surface we delete an annulus $S_i^1\times I$. We get two disjoint surfaces whose boundaries consist of two circles for each surface. We want to connect these circles by annuli embedded in $Y$. There are several ways to do this: One possibility is to connect the circles from the same surface. In this way we simply get back the surfaces $\Sigma_0$ and $\Sigma_1$. Another possibility is to connect the boundary circles from different surfaces. If this is possible we get an embedded connected surface of genus $g(\Sigma_0)+g(\Sigma_1)-1$ representing the class $[\Sigma_0]+[\Sigma_1]$. 

The construction works if we can find an embedded annulus $\Delta$ in $Y$ that intersects the surfaces $\Sigma_0$ and $\Sigma_1$ precisely in the circles $S^1_0$ and $S^1_1$. We also need a nowhere vanishing normal vector field along $\Delta$ that at the ends of $\Delta$ is tangential to the surfaces $\Sigma_0$ and $\Sigma_1$. The annuli connecting the four boundary circles are then constructed as normal push-offs of the annulus $\Delta$.  

\begin{lem}
There exists an embedded annulus $\Delta$ connecting the tori $F$ and $R$ that satisfies the necessary assumptions for the circle sum in \cite{LiLi4}.
\end{lem}
\begin{proof}
The elliptic surface $X=E(n)$ is obtained as a fibre sum of $E(n-1)$ and $E(1)$ along a general fibre. Let $S^1\times S^1\times D^2$ denote a tubular neighbourhood of the fibre in one of the summands. We think of $D^2$ as the unit disk in the complex plane and let $I$ denote the interval $[\frac {1}{2},1]$ along the real axis. In forming the fibre sum we delete the open tubular neighbourhood of radius $\frac{1}{4}$ of the general fibre in the centre of the tubular neighbourhood. The fibre $F$ in $X$ is realized as $S^1\times S^1\times \{\frac{1}{2}\}$ while the rim torus $R$ is $S^1\times \{\ast\}\times \partial D^2$.  Consider the annulus $\Delta=S^1\times \{\ast\}\times I$. It intersects the tori $F$ and $R$ precisely in the circles $S_F^1=S^1\times\{\ast\}\times \{\frac{1}{2}\}$ and $S_R^1=S^1\times\{\ast\}\times \{1\}$. The tangent bundle of $S^1\times S^1\times I\times \partial D^2$ is canonically trivial. Let $v_F$ be a unit tangent vector to $S^1$ in the point $\ast$ and $v_R$ a unit tangent vector to $\partial D^2$ in $1$. Then
\begin{equation*}
e_F=(0,v_F,0,0)\quad\text{along $S^1_F$}
\end{equation*}
and 
\begin{equation*}
e_R=(0,0,0,v_R)\quad\text{along $S^1_R$}
\end{equation*}
are framings of the circles $S_F^1$ and $S_R^1$ inside the tori. Consider the normal vector field along the annulus $\Delta$, given on $S^1\times \{\ast\}\times t$ by
\begin{equation*}
e=(0,(2-2t)v_F,0,(2t-1)v_R).
\end{equation*}
This is equal to the framings $e_F$ and $e_R$ on the boundary and is the required framing of the annulus.
\end{proof}
This construction allows us to circle sum $F$ and $R$. A similar, but easier construction allows us to circle sum $|\alpha|$ parallel copies of $F$ and $|\gamma|$ parallel copies of $R$ with a suitable orientation to get embedded tori $\Sigma_0$ and $\Sigma_1$ representing the classes $\alpha F$ and $\gamma R$. The torus $\Sigma_0$ contains as an open subset a copy of the torus $F$ with an annulus deleted, and similarly for $\Sigma_1$. Circle summing $\Sigma_0$ and $\Sigma_1$ along these subsets we get an embedded torus representing the class $\alpha F+\gamma R$. This construction proves:
\begin{thm}
Let $X$ be an elliptic surface without multiple fibres. Then any homology class $A$ orthogonal to $K$ and of self-intersection zero is represented by an embedded torus.
\end{thm}
This is clearly the minimal possible genus allowed by the adjunction inequality if the class $A$ is non-zero. The same method can be used to prove the following generalization:
\begin{thm}\label{main thm without multiple}
Let $X$ be an elliptic surface without multiple fibres. Suppose $A$ is a non-zero homology class orthogonal to $K$ such that $A^2=2c-2$ with $c\geq 0$. Then $A$ is represented by a surface of genus $c$ in $X$. This is the minimal possible genus.
\end{thm}
\begin{proof}
The cases $c=0$ and $c=1$ have been proved above. We can assume that $c\geq 2$. The assumptions imply that there exists a self-diffeomorphism of $X$ mapping $A$ to
\begin{equation*}
A'=\alpha F+ \gamma R+ \delta T,
\end{equation*}
where $\gamma$ and $\delta$ are positive with $\gamma\delta=c-1$. We circle sum $|\alpha|$ parallel copies of $F$ with a suitable orientation to get a torus $\Sigma_0$ representing $\alpha F$. Taking circle sums of parallel copies of the tori $R$ and $T$ we get tori representing $\gamma R$ and $\delta T$ that intersect transversely in $\gamma\delta$ points. Smoothing these intersections we get a surface $\Sigma_1$ of genus $\gamma\delta +1=c$. This surface contains as an open subset a copy of the torus $R$ with an annulus and $\delta$ points deleted. We circle sum the surface $\Sigma_1$ to the torus $\Sigma_0$ to get an embedded surface of genus $c$ representing $A'$.
\end{proof}

\subsection*{Acknowledgements} I would like to thank T.~Jentsch and D.~Kotschick for reading a preliminary version of this paper and the referee for helpful comments.

\bibliographystyle{amsplain}

\bigskip
\bigskip

\end{document}